\documentclass{amsart}

\usepackage{amsthm}
\usepackage[leqno]{amsmath}
\usepackage{latexsym,amsfonts,amssymb}
\usepackage[all]{xy} \SelectTips{eu}{} \SilentMatrices
\usepackage{hyperref}

\newcommand{\numberseries}{\bfseries}   

\newlength{\thmtopspace}                
\newlength{\thmbotspace}                
\newlength{\thmheadspace}               
\newlength{\thmindent}                  

\renewcommand{\subparagraph}{\vspace*{\thmbotspace}}

\newtheoremstyle{bfupright head,slanted body}
                {\thmtopspace}{\thmbotspace}
                {\slshape}{\thmindent}{\bfseries}{.}{\thmheadspace}
                {{\numberseries \thmnumber{#2\;}}\thmnote{#3}}

\newtheoremstyle{bfupright head,upright body}
                {\thmtopspace}{\thmbotspace}
                {\upshape}{\thmindent}{\bfseries}{.}{\thmheadspace}
                {{\numberseries \thmnumber{#2\;}}\thmnote{#3}}

\newtheoremstyle{bfit head,upright body}
                {\thmtopspace}{\thmbotspace}
                {\upshape}{\thmindent}{\upshape}{.}{\thmheadspace}
                {{\numberseries\thmnumber{#2\;}}
                {\bfseries\itshape\thmnote{\negthickspace#3}}}

\newtheoremstyle{it head,upright body}
                {\thmtopspace}{\thmbotspace}
                {\upshape}{\thmindent}{\upshape}{.}{\thmheadspace}
                {{\numberseries\thmnumber{#2\;}}
                {\itshape\thmnote{\negthickspace#3}}}

\newtheoremstyle{fixed bf head,slanted body}
                {\thmtopspace}{\thmbotspace}{\slshape}
                {\thmindent}{\bfseries}{.}{\thmheadspace}
                {{\numberseries \thmnumber{#2\;}}\thmname{#1}\thmnote{ (#3)}}

\newtheoremstyle{fixed bf head,upright body}
                {\thmtopspace}{\thmbotspace}{\upshape}
                {\thmindent}{\bfseries}{.}{\thmheadspace}
                {{\numberseries \thmnumber{#2\;}}\thmname{#1}\thmnote{ (#3)}}

\newtheoremstyle{fixed bfit head,upright body}
                {\thmtopspace}{\thmbotspace}{\upshape}
                {\thmindent}{\bfseries\itshape}{.}{\thmheadspace}
                {{\numberseries \thmnumber{#2\;}}\thmname{#1}\thmnote{ (#3)}}

\newtheoremstyle{sc head,small body}
                {\thmtopspace}{\thmbotspace}
                {\small\upshape}{\thmindent}{\scshape}{.}{\thmheadspace}
                {\thmname{#1}}

\newtheoremstyle{numbered paragraph}
                {\thmtopspace}{\thmbotspace}{\upshape}
                {\thmindent}{\upshape}{}{0pt}
                {{\numberseries \thmnumber{#2\;}}}

\newtheoremstyle{unnumbered paragraph}
                {\thmtopspace}{\thmbotspace}{\upshape}
                {\parindent}{\upshape}{}{0pt}

  \theoremstyle{bfupright head,slanted body}
  \newtheorem{res}{}[section]             \newtheorem*{res*}{}

\theoremstyle{sc head,small body}

\theoremstyle{fixed bf head,slanted body}
\newtheorem{thm}[res]{Theorem}          \newtheorem*{thm*}{Theorem}
\newtheorem{prp}[res]{Proposition}      \newtheorem*{prp*}{Proposition}
\newtheorem{cor}[res]{Corollary}        \newtheorem*{cor*}{Corollary}
\newtheorem{lem}[res]{Lemma}            \newtheorem*{lem*}{Lemma}

\theoremstyle{fixed bf head,upright body}
       \newtheorem*{dfn*}{Definition}
\newtheorem{obs}[res]{Observation}      \newtheorem*{obs*}{Observation}
\newtheorem{rmk}[res]{Remark}           \newtheorem*{rmk*}{Remark}
\newtheorem{exa}[res]{Example}          \newtheorem*{exa*}{Example}

\newlength{\thmlistleft}        
\newlength{\thmlistright}       
\newlength{\thmlistpartopsep}   
\newlength{\thmlisttopsep}      
\newlength{\thmlistparsep}      
\newlength{\thmlistitemsep}     

\newcounter{eqc} 
  {\end{list}}%

\newcounter{prt}
\newenvironment{prt}{\begin{list}{\upshape (\alph{prt})}%
    {\usecounter{prt}%
      \setlength{\leftmargin}{\thmlistleft}%
      \setlength{\labelwidth}{\thmlistleft}%
      \setlength{\rightmargin}{\thmlistright}%
      \setlength{\partopsep}{\thmlistpartopsep}%
      \setlength{\topsep}{\thmlisttopsep}%
      \setlength{\parsep}{\thmlistparsep}%
      \setlength{\itemsep}{\thmlistitemsep}}}%
  {\end{list}}%

\renewcommand{\eqref}[1]{(\pgref{eq:#1})}
\newcommand{\pgref}[1]{\ref{#1}}
\newcommand{\thmref}[2][Theorem~]{#1\pgref{thm:#2}}
\newcommand{\corref}[2][Corollary~]{#1\pgref{cor:#2}}
\newcommand{\prpref}[2][Proposition~]{#1\pgref{prp:#2}}
\newcommand{\lemref}[2][Lemma~]{#1\pgref{lem:#2}}
\newcommand{\obsref}[2][Observation~]{#1\pgref{obs:#2}}
\newcommand{\exaref}[2][Example~]{#1\pgref{exa:#2}}
\newcommand{\rmkref}[2][Remark~]{#1\pgref{rmk:#2}}
\newcommand{\secref}[2][Section~]{#1\ref{sec:#2}}
\newcommand{\thmcite}[2][?]{\cite[thm.~#1]{#2}}
\newcommand{\corcite}[2][?]{\cite[cor.~#1]{#2}}
\newcommand{\prpcite}[2][?]{\cite[prop.~#1]{#2}}
\newcommand{\lemcite}[2][?]{\cite[lem.~#1]{#2}}

\numberwithin{equation}{res}

\newcommand{\set}[2][\,]{\{#1 #2 #1\}}
\newcommand{\setof}[3][\,]{\{#1#2 \mid #3#1\}}
\newcommand{\NN}{\mathbb{N}}

\newcommand{\qtext}[1]{\quad\text{#1}\quad}
\newcommand{\qand}{\qtext{and}}
 \newcommand{\deq}{\:=\:}
 \newcommand{\f}{\varphi}
\newcommand{\m}{\mathfrak{m}}
\newcommand{\n}{\mathfrak{n}}
\newcommand{\q}{\mathfrak{q}}
\newcommand{\mult}[1]{\operatorname{e}(#1)}

\newcommand{\is}{\cong}
\newcommand{\qis}{\simeq}
\renewcommand{\le}{\leqslant}
\renewcommand{\ge}{\geqslant}
\newcommand{\onto}{\twoheadrightarrow}
\newcommand{\into}{\hookrightarrow}
\newcommand{\lra}{\longrightarrow}
\newcommand{\xra}[2][]{\xrightarrow[#1]{\;#2\;}}
\newcommand{\dra}[2]{\xra{\dif[#1]{#2}}}
\newcommand{\poly}[2][k]{#1[#2]}
\newcommand{\pows}[2][k]{#1[\mspace{-2.3mu}[#2]\mspace{-2.3mu}]}
\newcommand{\Rm}{(R,\m)}
\newcommand{\Rmk}{(R,\m,k)}
\newcommand{\Sn}{(S,\n)}
\newcommand{\Rhat}{\widehat{R}}
\newcommand{\Shat}{\widehat{S}}
\newcommand{\mapdef}[4][\rightarrow]{\nobreak{#2\colon #3 #1 #4}}
 \newcommand{\dif}[2][]{{\partial}_{#2}^{#1}}
\newcommand{\HH}[2][]{\operatorname{H}^{#1\!}#2}
\newcommand{\HHp}[2][]{\operatorname{H}^{#1}(#2)}

\newcommand{\Shift}[2][]{\mathsf{\Sigma}^{#1}{#2}}
\newcommand{\Shiftp}[2][]{(\Shift[#1]{#2})}
\newcommand{\dimR}{\operatorname{dim}R}
\renewcommand{\dim}[2][R]{\operatorname{dim}_{#1}#2}
\newcommand{\edim}[1]{\operatorname{edim}#1}
\newcommand{\codim}[1]{\operatorname{codim}#1}
 \newcommand{\dpt}[2][R]{\operatorname{depth}_{#1}#2}
\newcommand{\lgt}[2][R]{\operatorname{length}_{#1}#2}

\newcommand{\pd}[2][R]{\operatorname{proj.\!dim}_{#1}#2}

\newcommand{\RHom}[3][R]{\operatorname{\mathbf{R}Hom}_{#1}(#2,#3)}
\newcommand{\Ext}[4][R]{\operatorname{Ext}_{#1}^{#2}(#3,#4)}
\newcommand{\tp}[3][R]{\nobreak{#2\otimes_{#1}#3}}

\newcommand{\Ltp}[3][R]{\nobreak{#2\otimes_{#1}^{\mathbf{L}}#3}}

\newcommand{\Tor}[4][R]{\operatorname{Tor}^{#1}_{#2}(#3,#4)}

\def\urltilda{\kern -.15em\lower .7ex\hbox{\~{}}\kern .04em} 

\makeatletter
\def\@nobreak@#1{\mathchoice%
  {\nobreakdef@\displaystyle\f@size{#1}}%
  {\nobreakdef@\nobreakstyle\tf@size{\firstchoice@false #1}}%
  {\nobreakdef@\nobreakstyle\sf@size{\firstchoice@false #1}}%
  {\nobreakdef@\nobreakstyle\ssf@size{\firstchoice@false #1}}%
  \check@mathfonts}%
\def\nobreakdef@#1#2#3{\hbox{{%
                    \everymath{#1}%
                    \let\f@size#2\selectfont%
                    #3}}}%
\makeatother

\newenvironment{condition}[1][1.5em]{\begin{list}{}%
    {\setlength{\leftmargin}{#1}\setlength{\rightmargin}{#1}%
      \setlength{\partopsep}{0pt}%
      \setlength{\topsep}{0.6\thmbotspace}%
      \setlength{\parsep}{0pt}%
      \setlength{\itemsep}{0pt}}
      \item[]\small}
    {\end{list}}%

 \newcommand{\ab}[1]{b_{#1}}
 \newcommand{\ac}{\textsc{(ac)}}
 \newcommand{\uac}{\textsc{(uac)}}
 \newcommand{\acuac}{\textsc{(ac/uac)}}
 
 \newcommand{\syz}[3][R]{\operatorname{syz}^#1_{#2}(#3)}
 \newcommand{\x}{\pmb{x}}

\begin{document}

 \title[Vanishing of cohomology over {C}ohen--{M}acaulay
 rings]{Vanishing of cohomology\\ over {C}ohen--{M}acaulay rings}

 \author{Lars Winther Christensen}

 \address{Department of Math.\ and Stat., Texas Tech University,
   Lubbock, TX 79409, U.S.A.}

 \email{lars.w.christensen@ttu.edu}

 \urladdr{http://www.math.ttu.edu/\urltilda lchriste}

 \author{Henrik Holm}

 \address{Department of Basic Sciences and Environment, University of
   Copenhagen, Thorvaldsensvej 40, DK-1871 Frederiksberg C, Denmark}

 \email{hholm@life.ku.dk}

 \urladdr{http://www.matdat.life.ku.dk/\urltilda hholm}

 \date{23 February 2012}

 \keywords{AC ring, AB ring, Auslander's condition, Ext-index,
   c.i.~local homomorphism}

 \subjclass[2010]{13D07, 13H10}

\begin{abstract} 
  A 2003 counterexample to a conjecture of Auslander brought attention
  to a family of rings---colloquially called AC rings---that satisfy a
  natural condition on vanishing of cohomology. Several results attest
  to the remarkable homological properties of AC rings, but their
  definition is barely operational, and it remains unknown if they
  form a class that is closed under typical constructions in ring
  theory. In this paper, we study transfer of the AC property along
  local homomorphisms of Cohen--Macaulay rings. In particular, we show
  that the AC property is preserved by standard procedures in local
  algebra. Our results also yield new examples of Cohen--Macaulay AC
  rings.
\end{abstract}

\maketitle

\section*{Introduction}

\noindent
Vanishing of Ext groups and functors play a crucial role in the study
of rings and their modules. For a fixed module $M$, vanishing of
$\Ext[]{n}{M}{N}$ for all modules $N$ and integers $n \gg 0$ says that
$M$ has finite projective dimension, in which case the functors
$\Ext[]{n}{M}{-}$ vanish for all $n > \pd[]{(M)}$. Finiteness of
$\pd{(M)}$ for all finitely generated $R$-modules is a lot to ask from
a ring---in the noetherian case it means that $R$ is regular.
Auslander conjectured that for every finitely generated module $M$,
also those of infinite projective dimension, there would be a sort of
an upper bound for non-vanishing of the groups $\Ext[]{n}{M}{N}$; see
\cite[ch.~V]{mas1}. The exact conjecture is that every Artin algebra
$R$ satisfies the following:

\begin{condition}
  \ac\ For every finitely generated $R$-module $M$ there
  exists an integer $\ab{M} \ge 0$ such that for every finitely
  generated $R$-module $N$ one has: $\Ext{n}{M}{N}=0$ for $n \gg 0$
  implies $\Ext{n}{M}{N}=0$ for $n > \ab{M}$.
\end{condition}

\noindent An integer $\ab{M}$ with this property is called an
\emph{Auslander bound} for $M$; the least such bound may be perceived
as a ``latent projective dimension'' of $M$.

An example, discovered by Jorgensen and \c{S}ega \cite{DAJLMS04},
disproved Auslander's conjecture. However, several classes of rings do
satisfy Auslander's condition \ac, and remarkable homological
properties that flow from \ac\ have been uncovered, for example, in
work of Huneke and Jorgensen \cite{CHnDAJ03}; see also \cite{LWCHHlc}.

The counterexample in \cite{DAJLMS04} is a commutative Gorenstein
local ring, which is a finite dimensional algebra, and even
Koszul. Nevertheless, the condition \ac\ is satisfied by all
commutative local rings that are complete intersection or Golod, and
that is striking: with regard to homological characteristics these two
classes of rings usually belong on opposite sides of the spectrum.

Among commutative noetherian local rings, those that satisfy \ac\ are
emerging as a family with intriguing homological properties, albeit
one whose position relative to the traditional classes of rings is not
easily described. In this paper, we give new examples---explicit and
abstract---of Cohen--Macaulay local rings that satisfy \ac.

\begin{center}
  $*\:*\:*$
\end{center}

From this point on, all rings considered are commutative, noetherian,
and local. In particular, $(R,\m)$ is such a ring, and $\Rhat$ denotes
its $\m$-adic completion.

Most Cohen--Macaulay rings of finite CM type are Golod and, therefore,
known from \prpcite[1.4]{DAJLMS04} to satisfy even a uniform version
of Auslander's condition:
\begin{condition}
  \uac\ There exists an integer $b \ge 0$ such that for all
  finitely generated $R$-modules $M$ and $N$ one has:
  $\Ext{n}{M}{N}=0$ for $n\gg 0$ implies $\Ext{n}{M}{N}=0$ for~$n >
  b$.
\end{condition}
\noindent In \secref{cm} we show by a direct argument that every
Cohen--Macaulay ring of finite CM type satisfies \uac.

An explicit collection of new examples of rings that satisfy \uac\ is
constructed via results about preservation of the \ac\ and \uac\
properties under standard operations in local algebra---such as
completion and reduction modulo regular elements. These results were
advertised in \cite[rmk.~5.7]{LWCHHlc}; they are proved in
\secref{stab}, and the examples follow in \secref{example}.

The canonical maps \mbox{$R \into \Rhat$} and \mbox{$R\onto R/(x)$},
where $x$ is a regular element, are archetypes of local homomorphisms
of finite flat dimension; they are even c.i.\ homomorphisms in the
sense of Avramov~\cite{LLA99}. A classical chapter of local algebra
studies transfer of ring theoretic properties along local
homomorphisms. In \secref{homos} we add to it as we prove,
essentially, that the \ac\ and \uac\ properties descend along
homomorphisms of finite flat dimension and ascend along
c.i. homomorphisms.

\section{Maximal Cohen--Macaulay modules}
\label{sec:cm}

For a finitely generated $R$-module $M$, the $n$th syzygy in its
minimal free resolution is denoted $\syz{n}{M}$. For $i > n \ge 0$
there are isomorphisms:
\begin{equation}
  \label{eq:is1}
  \Ext{i-n}{\syz{n}{M}}{N} \is \Ext{i}{M}{N}.
\end{equation}

Let $R$ be Cohen--Macaulay of dimension $d$. A finitely generated
$R$-module $M$ is called \emph{maximal Cohen--Macaulay}, abbreviated
MCM, if the equality $\dpt{M}=d$ holds. Every $d$th syzygy of a
finitely generated $R$-module is either $0$ or an MCM module, and
every syzygy of an MCM module is an MCM module; see \prpcite[(1.16)
and (1.3)]{yos}. These facts together with the theory of MCM
approximations, which is due to Auslander and Buchweitz
\cite{MAsROB89}, are central to this section.

\begin{rmk}
  \label{rmk:mcmab}
  Let $R$ be Cohen--Macaulay of dimension $d$. Dimension shifting
  \eqref{is1} shows that $R$ satisfies \ac\ if and only if every MCM
  $R$-module has an Auslander bound. Every MCM module is a (finite)
  direct sum of indecomposable MCM modules, so $R$ satisfies \ac\ if
  and only if every indecomposable MCM $R$-module has an Auslander
  bound. In particular, $R$ satisfies \uac\ if and only if the
  indecomposable MCM $R$-modules have a common Auslander bound.
\end{rmk}

A Cohen--Macaulay ring $R$ is said to be of \emph{finite CM type} if
there are only finitely many indecomposable MCM $R$-modules, up to
isomorphism.

\begin{thm}
  \label{thm:cm}
  A Cohen--Macaulay local ring of finite CM type satisfies \uac.
\end{thm}

\begin{proof}
  Let $R$ be Cohen--Macaulay of finite CM type.  By \rmkref{mcmab}, it
  suffices to prove that each of the finitely many indecomposable MCM
  $R$-modules has an Auslander bound. For an MCM $R$-module $X$, let
  $n(X)$ be the number of distinct, up to isomorphism, indecomposable
  MCM $R$-modules that occur as direct summands of any one of the
  syzygies $\syz{i}{X}$, for $i \ge 0$.

  Let $M$ be an indecomposable MCM $R$-module. If $n(M)$ is $1$, then
  $M$ is either free of rank $1$ or a direct summand of every syzygy
  module $\syz{i}{M}$ for $i \ge 0$, whence it has Auslander bound
  $0$; cf.~\eqref{is1}. Let $t \ge 1$ and assume that every MCM
  $R$-module $X$ with $n(X) \le t$ has an Auslander bound. Assume
  $n(M) = t+1$. If $M$ occurs as a direct summand of one of its
  syzygies, say, $\syz{l}{M}$ for some $l \ge 1$, then---by uniqueness
  of minimal free resolutions---$M$ occurs as a summand of infinitely
  many syzygies, namely $\syz{lm}{M}$ for all $m\in\NN$; whence $M$
  has Auslander bound $0$. If $M$ does not occur as a direct summand
  of $\syz{i}{M}$ for any $i \ge 1$, then one has $n(\syz{1}{M}) \le
  t$. By the induction hypothesis, $\syz{1}{M}$ has an Auslander
  bound, whence $M$ has one, again by dimension shifting \eqref{is1}.
\end{proof}

The \emph{embedding dimension} of $R$, $\edim{R}$, is the minimal
number of generators of its maximal ideal, and the \emph{codimension}
is the difference $\codim{R} = \edim{R} - \dimR$. The multiplicity
$\mult{R}$ of a Cohen--Macaulay ring is at least $\codim{R} +1 $; if
equality holds, then $R$ is said to have \emph{minimal multiplicity.}

\begin{rmk}
  Eisenbud and Herzog \cite{DEsJHr88} raise the question whether every
  complete Cohen--Macaulay ring of finite CM type and dimension at
  least $2$ has minimal multiplicity. Cohen--Macaulay rings of minimal
  multiplicity are Golod, see Avramov~\prpcite[5.2.4]{ifr}, and
  satisfy \uac\ for a strong reason, see \prpcite[1.4]{DAJLMS04}.
\end{rmk}

\begin{obs}
  \label{obs:Syz}
  Let $R$ be Cohen--Macaulay and assume that it has a dualizing
  module.  By \thmcite[A]{MAsROB89}, every finitely generated
  $R$-module $N$ has a maximal Cohen--Macaulay approximation: an exact
  sequence of finitely generated $R$-modules,%
  \begin{equation*}
    0 \lra I \lra Y \lra N \lra 0,
  \end{equation*}
  where $Y$ is MCM, and $I$ is of finite injective dimension.  If $X$ is
  an MCM module, then one has $\Ext{n}{X}{I} =0$ for $n\ge 1$; see
  Yoshino \corcite[(1.13)]{yos} and \corcite[6.4]{MAsROB89}. Thus, for
  all $n \ge 1$ there are isomorphisms
  \begin{equation*}
    \Ext{n}{X}{Y} \is \Ext{n}{X}{N}.
  \end{equation*}
  From \eqref{is1} and the displays above, it is not hard to see that
  it suffices to verify the conditions \ac\ and \uac\ on MCM
  modules. That is, $R$ satisfies \ac\ if and only if it satisfies:
  \begin{condition}
    For every MCM $R$-module $X$ there exists $\ab{X} \ge 0$ such that
    for every MCM $R$-module $Y$ one has: $\Ext{n}{X}{Y}=0$ for $n\gg
    0$ implies $\Ext{n}{X}{Y}=0$ for~$n> \ab{X}$.
  \end{condition}
  And $R$ satisfies \uac\ if and only if it satisfies:
  \begin{condition}
    There exists an integer $b \ge 0$ such that for all MCM
    $R$-modules $X$ and $Y$ one has: $\Ext{n}{X}{Y}=0$ for $n\gg 0$
    implies $\Ext{n}{X}{Y}=0$ for $n> b$.
  \end{condition}
\end{obs}

\section{Completion and reduction modulo regular sequences}
\label{sec:stab}

Let $R$ be Cohen--Macaulay; we can now show that $R$ and $\Rhat$
satisfy \ac\ simultaneously and---in a generalization of
\prpcite[3.3.(1)]{CHnDAJ03}---that $R$ and $R/(x)$ satisfy \ac\
simultaneously, if $x$ is a regular element.

\begin{lem}
  \label{lem:regelt}
  Let $x$ be a regular element in $R$. If $R$ satisfies \acuac, then
  $R/(x)$ satisfies \acuac.
\end{lem}

\begin{proof}
  Let $M$ and $N$ be finitely generated $R/(x)$-modules. The change of
  rings spectral sequence \cite[XVI.\S$5.(2)_3$]{careil}
  \begin{equation*}
    \mathrm{E}_2^{p,q} = \Ext[R/(x)]{q}{\Tor{p}{R/(x)}{M}}{N} \Rightarrow
    \Ext{p+q}{M}{N}
  \end{equation*}
  has zero differentials on the second and all subsequent pages. For
  each $n\ge 0$ it follows from \prpcite[5.5]{careil} that there is an
  exact sequence
  \begin{equation*}
    0 \to \Ext[R/(x)]{n-1}{M}{N} \to \Ext{n}{M}{N} \to
    \Ext[R/(x)]{n}{M}{N}  \to 0.
  \end{equation*}
  If $b$ is an Auslander bound for $M$ as an $R$-module, then it
  follows that $\max\{0,b-1\}$ is an Auslander bound for $M$ as an
  $R/(x)$-module.
\end{proof}
\noindent
A slightly different argument for the result above appeared in the
proof of \prpcite[3.2.(1)]{CHnDAJ03}; the converse for Cohen--Macaulay
local rings is established in \corref[]{cmreg}.

The remaining proofs in this section use computations in the
derived category over $R$; first we recall some notation.  A complex
of $R$-modules is graded cohomologically,
\begin{equation*}
  M \deq \cdots \lra M^{n-1} \dra{n-1}{M} M^{n} \dra{n}{M}
  M^{n+1} \lra \cdots.
\end{equation*}
The suspension of $M$ is the complex $\Shift{M}$ with $\Shiftp{M}^n =
M^{n+1}$ and $\dif{\Shift{M}} = - \dif{M}$. The $n$th cohomology
module of $M$ is denoted \mbox{$\HH[n]{\mspace{-2mu}M}$}. Isomorphisms
in the derived category over $R$ are marked by the symbol~$\qis$; they
induce isomorphisms at the level of homology.  We use the standard
notation, $\RHom{-}{-}$ and $\Ltp{-}{-}$, for the right derived Hom
functor and the left derived tensor product functor. For all
$R$-modules $M$ and $N$ and all integers $n$ there are isomorphisms
\begin{equation*}
  \Ext{n}{M}{N} \is \HH[n]{\RHom[A]{M}{N}} \qand \Tor{n}{M}{N} \is
  \HHp[-n]{\Ltp{M}{N}}.
\end{equation*}

In the next work-hose lemma, the ring $S$ need not be local.

\begin{lem}
  \label{lem:descent}
  Let $\Rm$ be Cohen--Macaulay and assume that it has a dualizing
  module.  Let $S$ be an $R$-algebra of finite flat dimension and with
  $S/\m S \ne 0$. If $S$ satisfies \acuac, then $R$ satisfies
  \acuac.
\end{lem}

\begin{proof}
  By \obsref{Syz} it suffices to consider cohomology of MCM
  $R$-modules. For such modules $X$ and $Y$ there are isomorphisms in
  the derived category over $R$,
  \begin{align*}
    \Ltp{\RHom{X}{Y}}{S} 
    & \qis \RHom[S]{\Ltp{X}{S}}{\Ltp{Y}{S}} \\
    & \qis \RHom[S]{\tp{X}{S}}{\tp{Y}{S}}.
  \end{align*}
  The first isomorphism is tensor evaluation, see Avramov and
  Foxby~\lemcite[4.4.(F)]{LLAHBF91}, and adjointness of Hom and
  tensor. The second isomorphism follows as one has $\Tor{n}{X}{S} =0=
  \Tor{n}{Y}{S}$ for $n\ge 1$; see \thmcite[(3.4)]{CFF-02}.  These
  isomorphisms and the assumption $S/\m S \ne 0$ yield:
  \begin{align*}
    \sup\setof{n}{\Ext{n}{X}{Y} \ne 0} =
    \sup\setof{n}{\Ext[S]{n}{\tp{X}{S}}{\tp{Y}{S}}\ne 0};
  \end{align*}
  see Foxby~\lemcite[2.1.(2)]{HBF77b}. Thus, if $b$ is an Auslander
  bound for the $S$-module $\tp{X}{S}$, then $b$ is also an Auslander
  bound for the $R$-module $X$.
\end{proof}

\begin{thm}
  \label{thm:completion}
  Let $\Rm$ be Cohen--Macaulay. Any one of the local rings
  \begin{equation*}
    R, \ \, \Rhat, \ \, \pows[R]{X}, \ \,\text{and} \ \, \poly[R]{X}_{(\m,X)}
  \end{equation*}
  satisfies \acuac\ if and only if they all satisfy \acuac.
\end{thm}

\begin{proof}
  The rings $R$, $\Rhat$, $\pows[R]{X}$, and $\poly[R]{X}_{(\m,X)}$
  are all local and Cohen--Macaulay.

  We first argue that $R$ and $\Rhat$ satisfy \acuac\
  simultaneously. If $\Rhat$ satisfies \acuac, then so does $R$ by
  \prpcite[5.5]{LWCHHlc}, as $\Rhat$ is a faithfully flat $R$-algebra.
  For the converse, assume that $R$ satisfies \acuac, and let
  $\x=x_1,\dots,x_d$ be a maximal $R$-regular sequence; its image in
  $\Rhat$ is a maximal $\Rhat$-regular sequence. There is an
  isomorphism of rings $R/(\x) \is \Rhat/(\x)$, and this ring is
  \acuac\ by \lemref{regelt}.  As $\Rhat$ has a dualizing module, it
  satisfies \acuac\ by \lemref{descent}.

  If $\pows[R]{X}$ or $\poly[R]{X}_{(\m,X)}$ satisfies \acuac, then
  so does $R$ by \lemref{regelt}. For the converse, assume that $R$
  and, therefore, $\Rhat$ satisfies \acuac. The local ring
  $\pows[\Rhat]{X}$, which is Cohen--Macaulay and has a dualizing
  module, satisfies \acuac\ by \lemref{descent} applied to the
  surjection $\pows[\Rhat]{X} \onto \Rhat$. The isomorphisms
  \begin{equation*}
    \widehat{\poly[R]{X}_{(\m,X)}} \is \pows[\Rhat]{X} \is \widehat{\pows[R]{X}}
  \end{equation*}
  and the assertion established in the first part of the proof now
  show that $\pows[R]{X}$ and $\poly[R]{X}_{(\m,X)}$ satisfy \acuac.
\end{proof}

\begin{cor}
  \label{cor:cmreg}
  Let $R$ be Cohen--Macaulay and $\x=x_1,\dots,x_n$ be an $R$-regular
  sequence. The ring $R/(\x)$ satisfies \acuac\ if and only if $R$
  satisfies \acuac.
\end{cor}

\begin{proof}
  The ``if'' part is immediate from \lemref{regelt}. For the converse,
  assume that $R/(\x)$ satisfies \acuac. It is a Cohen--Macaulay
  ring, so by \thmref{completion} the completion
  \mbox{$\smash{\widehat{R/(\x)} \is \Rhat/(\x)}$} satisfies
  \acuac. As the image of $\x$ in $\Rhat$ is a regular sequence, the
  ring $\Rhat$ satisfies \acuac\ by \lemref{descent}, and then
  another application of \thmref[]{completion} shows that $R$
  satisfies \acuac.
\end{proof}

A Cohen--Macaulay ring reduces modulo a regular sequence to an
artinian ring, and by \corref{cmreg} the two rings satisfy \acuac\
simultaneously. In a clear analogy, the
conditions \ac\ and \uac\ can be verified on modules of finite length:

\begin{prp}
  Let $R$ be Cohen--Macaulay. It satisfies \ac\ if and only if it
  satisfies the condition:
  \begin{condition}
    For every $R$-module $K$ of finite length there exists an integer
    $\ab{K} \ge 0$ such that for every $R$-module $L$ of finite length
    one has: $\Ext{n}{K}{L}=0$ for $n\gg 0$ implies $\Ext{n}{K}{L}=0$
    for $n> \ab{K}$.
  \end{condition}
  Furthermore, $R$ satisfies \uac\ if and only if it satisfies the
  condition:
  \begin{condition}
    There exists an integer $b \ge 0$ such that for all $R$-modules
    $K$ and $L$ of finite length one has: $\Ext{n}{K}{L}=0$ for $n\gg
    0$ implies $\Ext{n}{K}{L}=0$ for $n> b$.
  \end{condition}
\end{prp}

\begin{proof}
  Assume that $R$ satisfies \acuac\ for modules of finite
  length. For $\Rhat$-modules $K$ and $L$ of finite length and
  integers $n \ge 0$ there are isomorphisms
  \begin{equation*}
    \Ext[\Rhat]{n}{K}{L} \is \Ext[\Rhat]{n}{\tp{\Rhat}{K}}{L} \is \Ext{n}{K}{L}.
  \end{equation*}
  As $R$-modules, $K$ and $L$ also have finite length; hence it
  follows that $\Rhat$ satisfies \acuac\ for modules of finite
  length. In view of \thmref{completion} we can now assume that $R$ is
  complete and, in particular, that it has a dualizing module.

  By \obsref{Syz} it is sufficient to prove that $R$ satisfies
  \acuac\ for MCM modules. Let $X$ and $Y$ be MCM $R$-modules, and
  let $\x=x_1,\dots,x_d$ be a maximal $R$-regular sequence. As $\x$ is
  regular on $X$ and $Y$, there are isomorphisms
  \begin{equation}
    \label{eq:aa}
    \tag{1}
    X/\x X \qis \Ltp{X}{R/(\x)} \qand Y/\x Y \qis \Ltp{Y}{R/(\x)}
  \end{equation}
  in the derived category over $R$. Nakayama's lemma,
  see~\lemcite[2.1.(2)]{HBF77b}, yields the first equality below,
  while the second follows from \lemcite[4.4.(F)]{LLAHBF91} and
  \eqref{aa}.
  \begin{equation}
    \label{eq:ab}
    \tag{2}
    \begin{split}
      \sup\setof{n}{\HH[n]{&\RHom{X}{Y} \ne 0}}\\ 
      &= \sup\setof{n}{\HHp[n]{\Ltp{\RHom{X}{Y}}{R/(\x)}}\ne 0}\\
      &= \sup\setof{n}{\HH[n]{\RHom{X}{Y/\x Y} \ne 0}}.
    \end{split}
  \end{equation}
  From \eqref{aa} and adjointness of Hom and tensor one gets
  \begin{equation}
    \label{eq:ac}
    \tag{3}
    \RHom{X/\x X}{Y/\x Y} \qis \RHom{X}{\RHom{R/(\x)}{Y/\x Y}}.
  \end{equation}
  The Koszul complex on $\x$ is a free resolution of $R/(\x)$, so
  there is an isomorphism
  \begin{equation}
    \label{eq:ad}
    \tag{4}
    \Shift[d]{\RHom{R/(\x)}{-}} \qis \Ltp{-}{R/(\x)}.
  \end{equation}
  The next equalities follow from \eqref{ab} and Nakayama's lemma,
  from \lemcite[4.4.(F)]{LLAHBF91} and \eqref{ad}, and from
  \eqref{ac}, respectively.
  \begin{align*}
    \sup\setof{n}{\HH[n]{&\RHom{X}{Y}}\ne 0}\\
    &= \sup\setof{n}{\HHp[n]{\Ltp{\RHom{X}{Y/\x Y}}{R/(\x)}}\ne 0}\\
    &= \sup\setof{n}{\HH[n]{\RHom{X}{\Shift[d]{\RHom{R/(\x)}{Y/\x
              Y}}}}\ne 0 }\\
    &= \sup\setof{n}{\HH[n]{\RHom{X/\x X}{Y/\x Y}}\ne 0 } -d.
  \end{align*}
  That is, one has
  \begin{align*}
    \sup\setof{n}{\Ext{n}{X}{Y}\ne 0}= \sup\setof{n}{\Ext{n}{X/\x
        X}{Y/\x Y}\ne 0 } -d.
  \end{align*}
  As the $R$-modules $X/\x X$ and $Y/\x Y$ have finite length, a
  straightforward argument finishes the proof, cf.~proof of
  \lemref{descent}.
\end{proof}

\section{An example}
\label{sec:example}

One can construct new examples of Cohen--Macaulay local rings that
satisfy \acuac\ from existing ones through the process of adjoining
power series variables and killing regular elements. To get examples
not covered in the literature, see the list \cite[A.1]{LWCHHlc}, the
new ring $R$ should not be Golod or complete intersection (c.i.); and
if $R$ is Gorenstein, then its codimension, $\codim{R}$, should be at
least $5$, and its multiplicity, $\mult{R}$, should be at least
$\codim{R}+3$. Furthermore, $R$ should not be of finite CM type,
cf.~\thmref{cm}.

\enlargethispage{\baselineskip}
\begin{prp}
  \label{prp:heitmann}
  Let $Q$ be a Cohen--Macaulay local ring and assume that it satisfies
  \acuac. For integers $d,n \ge 1$ and $s_i \ge 2$, the local ring
  \begin{equation*}
    R = \pows[Q]{X_1,\dots,X_d,Y_1,\dots,Y_n}/(Y_1^{s_1},\dots,Y_n^{s_n})
  \end{equation*}
  is Cohen--Macaulay with $\codim{R} = \codim{Q} + n$, and it
  satisfies \acuac. Moreover, the following hold:
  \begin{prt}
  \item $R$ is not of finite CM type.
  \item If $Q$ is not c.i., then $R$ is not c.i.\ and $R$ is not
    Golod.
  \item If $Q$ has infinite residue field and one has $s_i \ge 4$ for
    some $i\in\set{1,\ldots,n}$, then the inequality $\mult{R} \ge
    \codim{R} + 3$ holds.
  \end{prt}
\end{prp}

\begin{proof}
  Clearly, $R$ is Cohen--Macaulay with $\codim{R} = \codim{Q} + n$,
  and it is immediate from \thmref{completion} and \corref{cmreg} that
  $R$ satisfies \acuac.

  (a): Let $\q$ be the maximal ideal in $Q$. The image of $Y_1$ in
  $R_{(\q,Y_1,\dots,Y_n)}$ is a zero-divisor, so $R$ is not an
  isolated singularity; in particular, $R$ is not of finite CM type;
  see Huneke and Leuschke~\corcite[2]{CHnGJL02}.

  (b): Assume that $Q$ is not c.i. By \cite[(5.10)]{LLA99} the ring
  $R$ is not c.i. As $Q$ is singular, one has $\codim{R} =\codim{Q} + n
  \ge 2$. The $R$-module $R/(Y_1)$ has constant Betti numbers equal to
  $1$. It follows from~\thmcite[5.3.3.(5)]{ifr} that $R$ is not Golod.

  (c): Let $\mathfrak{a}$ be a minimal reduction of the maximal ideal $\q$ in
  $Q$. It is elementary to verify that $\mathfrak{b} =
  (\mathfrak{a},X_1,\ldots,X_d)$ is a reduction of $\m =
  (\q,X_1,\ldots,X_d,Y_1,\ldots,Y_n)$, the maximal ideal in
  $R$. Assume that $Q/\q$ is infinite. It follows that $\mathfrak{a}$ is
  minimally generated by $\dim[]{Q}$ elements and, therefore,
  $\mathfrak{b}$ is generated by $\dim[]{Q}+d=\dimR$ elements. Hence,
  $\mathfrak{b}$ is a minimal reduction of $\m$; see Swanson and
  Huneke~\prpcite[8.3.7 and cor.~8.3.5.(1)]{icirm}.  As $R$ is
  Cohen--Macaulay, \prpcite[11.2.2]{icirm} yields $\mult{R} =
  \lgt[]{R/\mathfrak{b}}$. If one has $s_i \ge 4$ for some $i$, then
  $(\m/\mathfrak{b})^3$ is non-zero, whence
  \begin{equation*}
    \lgt[]{R/\mathfrak{b}} \ge \edim{R/\mathfrak{b}} + 3 \ge
    \edim{R} - \dimR + 3 = \codim{R} + 3. \qedhere
  \end{equation*}
\end{proof}

We thank Louiza Fouli for clarifying for us the basics of multiplicities
that are used in the proof of part (c) above.

All assumptions on $Q$ in \prpref{heitmann} are satisfied by the ring
in \exaref{Q}; thus for all choices of $d,n \ge 1$ and $s_i \ge 2$ the
resulting ring $R$ in \prpref[]{heitmann} is not among the rings on
the list \cite[A.1]{LWCHHlc} nor is it of finite CM type.
  
\begin{exa}
  \label{exa:Q}
  Let $k$ be a field of characteristic $0$ and consider the local
  $k$-algebra
  \begin{equation*}
    Q = \poly{U,V,W}/(U^2-V^2,V^2 - W^2,UV,VW).
  \end{equation*}
  It has radical cube zero and $\edim{Q}=3$, so it satisfies \uac\ by
  \prpcite[1.1.(3)]{DAJLMS04}. Moreover, $Q$ is not Gorenstein; in
  particular, it is not c.i.
\end{exa}

Through a result of Heitmann \cite{RCH93}, the construction in
\prpref{heitmann} also provides examples of unique factorization
domains that satisfy \acuac.

\begin{obs}
  Let $Q$ be a complete Cohen--Macaulay local ring that satisfies
  \acuac. Assume that $Q$ is not c.i.\ and that no integer is a zero
  divisor in~$Q$; the ring in \exaref[]{Q} serves as an example.

  Let $R$ be as in \prpref{heitmann}. If $Q$ is not artinian or $d$ is
  at least $2$, then $R$ has depth at least $2$. By
  \thmcite[8]{RCH93}, there is a local unique factorization domain,
  $D$, with $\widehat{D} \is R$. It follows from \thmref{completion}
  that $D$ satisfies \acuac.  Of course, $R$ is also the completion
  of the ring
  $\poly[Q]{X_1,\dots,X_d,Y_1,\dots,Y_n}/(Y_1^{s_1},\dots,Y_n^{s_n})$
  localized at the irrelevant maximal ideal, but that ring is not a
  domain.
\end{obs}

\section{Local homomorphisms}
\label{sec:homos}

Let $\Rm$ and $\Sn$ be local rings, and let $\mapdef{\f}{R}{S}$ be a
\emph{local} homomorphism, that is, a ring homomorphism with $\f(\m)
\subseteq \n$. The \emph{semicompletion} of~$\f$~is its composite with
the embedding $S \into \Shat$; it is written
$\mapdef{\grave{\f}}{R}{\Shat}$. The semicompletion of $\f$ admits a
Cohen factorization; that is, there exist local homomorphisms
\begin{equation*}
  R \xra{\dot{\f}} R' \xra{\f'} \Shat,
\end{equation*}
such that $\grave{\f}=\f'\dot{\f}$, where $\dot{\f}$ is flat with
regular closed fiber $R/\m R'$, the ring $R'$ is complete, and $\f'$
is surjective; see Avramov, Foxby, and
Herzog~\thmcite[(1.1)]{AFH-94}. If $\f$ is of finite flat
dimension---that is, $S$ has finite flat dimension as an $R$-module
via $\f$---then the surjection $\f'$ is of finite projective
dimension; see~\cite[(3.3)]{AFH-94}.

A Gorenstein local ring satisfies \ac\ if and only if it satisfies
\uac; following \cite{CHnDAJ03} such a ring is called an \emph{AB
  ring}. The \ac\ and \uac\ properties descend along local
homomorphisms of finite flat dimension:

\begin{thm}
  Let $\mapdef{\f}{R}{S}$ be a local homomorphism of finite flat
  dimension.
  \begin{prt}
  \item If $S$ is Cohen--Macaulay and satisfies \acuac, then $R$ is
    Cohen--Macaulay and satisfies \acuac.
  \item If $S$ is an AB ring, then $R$ is an AB ring.
  \end{prt}
\end{thm}

\begin{proof}
  Assume that $S$ is Cohen--Macaulay and satisfies \acuac, then
  $\Shat$ has the same properties; see \thmref{completion}.  Consider
  a Cohen factorization $R \to R' \to \Shat$ of the semicompletion
  $\grave{\f}$.  By \cite[(3.10)]{AFH-94} the rings $R'$ and $R$ are
  Cohen--Macaulay. The ring $R'$ is complete, hence it has a dualizing
  module, and by \lemref{descent} it satisfies \acuac. Finally, $R$
  satisfies \acuac\ by \prpcite[5.5]{LWCHHlc}; this proves part
  (a). For part (b) it suffices to note that if $S$ is Gorenstein,
  then $R$ is Gorenstein; see Avramov and Foxby
  \thmcite[(2.4)]{LLAHBF90}.
\end{proof}

The notation $\Rmk$ specifies that $k$ is the residue field of the local
ring $\Rm$.

\begin{lem}
  \label{lem:ascent}
  Let $\mapdef{\f}{\Rmk}{(S,\n,k)}$ be a local homomorphism of
  complete $k$-algebras. If $R$ is Cohen--Macaulay and satisfies
  \acuac, then $\f$ has a Cohen factorization $R \to R' \to S$,
  where $R'$ is Cohen--Macaulay and satisfies \acuac.
\end{lem}

\begin{proof}
  Let $y_1,\ldots,y_n$ be a generating set for $\n$. As the field $k$
  is a Cohen ring for $S$, in the sense of \cite[(1.0.2)]{AFH-94}, the
  inclusion $k \into S$ extends by the assignment $Y_i \mapsto y_i$ to
  a surjective homomorphism $\pi\colon \pows[k]{Y_1,\ldots,Y_n} \onto
  S$. Set $R' = \pows[R]{Y_1,\ldots,Y_n}$, then $R'$ is complete and a
  flat local extension of $R$ with regular closed fiber $R'/\m R' \is
  \pows[k]{Y_1,\ldots,Y_n}$. The homomorphism $\f$ extends by the
  assignment $Y_i \mapsto y_i$ to a homomorphism $R' \to S$; it is
  surjective because $\pi$ is surjective. Thus, $R \into R' \onto S$
  is a Cohen factorization of $\f$. If $R$ is Cohen--Macaulay and
  satisfies \acuac, then $R'$ has the same properties; see
  \thmref{completion}.
\end{proof}

Following \cite[3.1]{LLA99}, a local homomorphism $\mapdef{\f}{R}{S}$
is called \emph{c.i.}, for complete intersection, if there is a Cohen
factorization $R \to R' \to \Shat$ of the semicompletion $\grave{\f}$
in which the kernel of the surjection $R' \onto S$ is generated by an
$R'$-regular sequence.

Here is an ascent result for the \ac\ and \uac\ properties:

\begin{thm}
  Let $\mapdef{\f}{\Rmk}{(S,\n,k)}$ be a c.i.\ local homomorphism of
  $k$-algebras.
  \begin{prt}
  \item If $R$ is Cohen--Macaulay and satisfies \acuac, then $S$ is
    Cohen--Macaulay and satisfies \acuac.
  \item If $R$ is an AB ring, then $S$ is an AB ring.
  \end{prt}
\end{thm}

\begin{proof}
  Assume that $R$ is Cohen--Macaulay and satisfies \acuac, then
  $\Rhat$ has the same properties; see \thmref{completion}.  By
  \lemref{ascent}, the semicompletion $\grave{\f}$ has a Cohen
  factorization $R \to R' \to \Shat$, in which $R'$ is Cohen--Macaulay
  and satisfies \acuac.  As $\f$ is c.i., the kernel of the
  surjection $R' \onto \Shat$ is generated by an $R'$-regular
  sequence; see \cite[(3.3)]{LLA99}. It follows from \corref{cmreg}
  that $\Shat$ is Cohen--Macaulay and satisfies \acuac\ and,
  therefore, $S$ is Cohen--Macaulay and satisfies \acuac; see
  \thmref{completion}. This proves part (a).

  For part (b), recall that c.i.\ homomorphisms are Gorenstein by
  \thmcite[(2.4)]{LLAHBF90}. Thus, if $R$ is Gorenstein, it then
  follows from \cite[(3.11)]{AFH-94} that $S$ is Gorenstein.
\end{proof}




\def\cprime{$'$}
 \providecommand{\MR}[1]{\mbox{\href{http://www.ams.org/mathscine%
t-getitem?mr=#1}{#1}}}
  \renewcommand{\MR}[1]{\mbox{\href{http://www.ams.org/mathscinet-getitem?mr=#%
1}{#1}}}
\providecommand{\bysame}{\leavevmode\hbox to3em{\hrulefill}\thinspace}
\providecommand{\MR}{\relax\ifhmode\unskip\space\fi MR }
\providecommand{\MRhref}[2]{%
  \href{http://www.ams.org/mathscinet-getitem?mr=#1}{#2}
}
\providecommand{\href}[2]{#2}


\end{document}